\documentclass[12pt]{article}

\usepackage{fullpage}
\usepackage{amsfonts}
\usepackage{enumitem}
\usepackage{amsmath}
\usepackage{array}
\usepackage{amssymb}
\usepackage{amsthm}
\usepackage{hyperref}
\usepackage{mathrsfs}
\usepackage{cite}
\usepackage{aliascnt}
\usepackage{centernot}
\usepackage{dsfont}
\usepackage{stmaryrd}
\usepackage{tikz-cd}
\DeclareMathOperator{\Irr}{Irr}
\DeclareMathOperator{\GL}{GL}

\DeclareMathOperator{\Sp}{Sp}

\DeclareMathOperator{\Char}{\mathsf{Char}}

\DeclareMathOperator{\geom}{{geom}}

\DeclareMathOperator{\lcm}{lcm}

\DeclareMathOperator{\Wild}{\mathsf{Wild}}
\DeclareMathOperator{\Tame}{\mathsf{Tame}}

\begin{document}
\setlength{\parindent}{15pt}

\newtheorem*{qst}{Question}
\newtheorem{thm}{Theorem}[section]
\newcommand{\thmautorefname}{Theorem}

\numberwithin{equation}{thm}

\newtheorem*{thm*}{Theorem}

\newaliascnt{prop}{thm}
\newtheorem{prop}[prop]{Proposition}
\aliascntresetthe{prop}
\newcommand{\propautorefname}{Proposition}

\newaliascnt{lem}{thm}
\newtheorem{lem}[lem]{Lemma}
\aliascntresetthe{lem}
\newcommand{\lemautorefname}{Lemma}

\newaliascnt{crl}{thm}
\newtheorem{crl}[crl]{Corollary}
\aliascntresetthe{crl}
\newcommand{\crlautorefname}{Corollary}
\theoremstyle{definition}
\newtheorem*{dfn}{Definition}
\newaliascnt{rmk}{thm}
\newtheorem{rmk}[rmk]{Remark}
\aliascntresetthe{rmk}
\newcommand{\rmkautorefname}{Remark}

\title{Hypergeometric sheaves and extraspecial groups in even characteristic}
\author{Lee Tae Young\\\small Department of Mathematics, Rutgers University, Piscataway, NJ 08854, USA\\ \small tae.young@binghamton.edu}

\maketitle

\begin{abstract}
We determine precisely which irreducible hypergeometric sheaves have an extraspecial normalizer in characteristic $2$ as their geometric monodromy groups. This resolves the last open case of the determination of local monodromy at $0$ of irreducible hypergeometric sheaves with finite geometric monodromy group.
\end{abstract}



\tableofcontents

\section{Introduction}
Let $p$ be a prime. The finite quotient groups of the \'etale fundamental group $\pi_1^{et}(\mathbb{G}_m/\overline{\mathbb{F}_p})$ are the finite groups generated by its Sylow $p$-subgroups together with at most one other element; this was conjectured by Abhyankar\cite{Ab57} and proved by Harbater\cite{Har}. In their recent series of papers, Katz, Rojas-Le\'on and Tiep realized many pairs of such finite group $G$ and its faithful complex representation $V$ as the geometric monodromy groups of explicitly written $\overline{\mathbb{Q}_\ell}$-local systems, usually hypergeometric sheaves, on $\mathbb{G}_m/\overline{\mathbb{F}_p}$.

In \cite{KT2, KT4, KRLT4}, they determined which pairs $(G,V)$ can be realized using irreducible hypergeometric sheaves. It would be desirable to achieve a complete classification of irreducible hypergeometric sheaves in terms of their monodromy groups. As hypergeometric sheaves are completely determined by its local monodromies, one can first ask which local monodromy at $0$ can be realized. Specifically, we want to see which triples $(G,V,g)$, where $g\in G$ is an element with certain properties which must be satisfied by the local monodromy at $0$ of a hypergeometric sheaf, can be realized by irreducible hypergeometric sheaves. This was also completed in \cite{KT2,KT4,KRLT4} for all but one type of groups: the so-called extraspecial normalizers in characteristic $2$. 

The goal of this paper is to completely answer this question, not only about the local monodromy at $0$ but also at $\infty$, and write down a complete classification of irreducible hypergeometric sheaves whose geometric monodromy group is a finite extraspecial normalizer in characteristic $2$. In \autoref{m2sp}, we classify the elements of extraspecial normalizers in characteristic $2$ which has odd order and has no eigenvalues of multiplicity exceeding $2$ on an $2^n$-dimensional representation extending the unique irreducible representation of the normal extraspecial $2$-subgroup of order $2^n$ of that dimension. This is one of several properties that must be satisfied by the local monodromies at $0$ and $\infty$ of irreducible hypergeometric sheaves. After taking some properties of hypergeometric sheaves into account, in \autoref{g_infty} we get a more restrictive list of possible local monodromies. We then study the hypergeometric sheaves defined by each combination of the local monodromies at $0$ and $\infty$, and determine precisely when it has finite geometric monodromy group. This is our main result \autoref{mainthm}, which says that these local systems are precisely those discovered and studied earlier by Pink, Sawin, Katz and Tiep in \cite{KT2,KT4}, where the isomorphism type of their geometric monodromy groups were determined. This completes the classification of the local systems in question.

\section{Preliminaries}
In this section, we recall some definitions and basic properties of the objects we will study. For more definitions and properties, we refer to \cite{K-ESDE} and \cite{KT2}.

For $x\in \{0,\infty\}$, we fix a choice of an inertia group $I(x)<\pi_1^{et}(\mathbb{G}_m/\overline{\mathbb{F}_p})$, and we denote by $P(x)$ the wild inertial group in $I(x)$. We also fix an element $\gamma_x$ of order prime to $p$ in $I(x)$ whose image in $I(x)/P(x)$ is a topological generator of this pro-cyclic group.

For a nontrivial additive character $\psi$ of $\mathbb{F}_p$ and multiplicative characters $\chi_1,\dots,\chi_D$ and $\rho_1,\dots,\rho_M$ of a finite extension of $\mathbb{F}_p$, there exists a hypergeometric sheaf of type $(D,M)$ $$\mathcal{H}=\mathcal{H}yp_\psi(\chi_1,\dots,\chi_D;\rho_1,\dots,\rho_M).$$ It is tame at $0$, and its local monodromy at $\infty$ (that is, the restriction to $I(\infty)$) is given by $\Tame\oplus \Wild$, where $\Tame$ is a $M$-dimensional tame representation and $\Wild$ is a $(D-M)$-dimensional irreducible representation of Swan conductor $1$. The \emph{geometric monodromy group} $G_{\geom}\leq \GL_D(\overline{\mathbb{Q}_\ell})$ of $\mathcal{H}$ is the Zariski closure of the image of $\pi_1^{et}(\mathbb{G}_m/\overline{\mathbb{F}_p})$ on $\mathcal{H}$. We denote by $J, Q\leq G_{\geom}$ and $g_x\in G_{\geom}$ the images of $I(\infty)$, $P(\infty)$ and $\gamma_x$ on $\mathcal{H}$ when the local system is clear from the context. The restriction $\Wild|_Q$ (or equivalently $\Wild|_{P(\infty)}$) decomposes into a direct sum of $W_0$ pairwise nonisomorphic irreducible representations, where $W_0$ is the $p'$-part of $\dim\Wild=D-M$. 

The characters $\chi_i$ are called the \emph{upstairs characters}, and $\rho_j$ are called the \emph{downstairs characters} of $\mathcal{H}$. $\mathcal{H}$ is irreducible if and only if no $\{\chi_1,\dots,\chi_D\}\cap\{\rho_1,\dots,\rho_M\}=\emptyset$. Also, $G_{\geom}$ can be finite only when $\chi_1,\dots,\chi_D$ are pairwise distinct and so are $\rho_1,\dots,\rho_M$. For all hypergeometric sheaves in the rest of this paper, we will always assume that these properties are satisfied. In particular, the element $g_0\in G_{\geom}<\GL_D(\overline{\mathbb{Q}_\ell})$ must be an $\mathsf{ssp}$-element and $g_\infty$ must be an $\mathsf{m2sp}$-element, as defined below. 

\begin{dfn}
An element $g\in \GL_n(\mathbb{C})$ is said to be 
\begin{enumerate}[label=(\roman*)]
\item an \emph{$\mathsf{ssp}$-element} if each eigenvalue of $g$ has multiplicity $1$, and
\item an \emph{$\mathsf{m2sp}$-element} if each eigenvalue of $g$ has multiplicity at most $2$.
\end{enumerate}
\end{dfn}

Katz and Tiep \cite[Theorem 5.2.9]{KT4} \cite{KT5} showed that most irreducible primitive hypergeometric sheaves (as representations of their $G_{\geom}$) satisfy the following condition.
\begin{dfn}
A group $G$ and its faithful complex representation $V$ (or equivalently, a subgroup $G\leq \GL_n(\mathbb{C})$) is said to satisfy \emph{condition $\mathbf{(S)}$} if $V$ is irreducible, primitive, tensor indecomposable, and not tensor induced. If in addition $\mathbf{Z}(G)$ is finite, then we say $(G,V)$ satisfies \emph{condition $\mathbf{(S+)}$}. 
\end{dfn}

By a result of Guralnick and Tiep \cite[Proposition 2.8]{GT2}, such finite groups, in particular the $G_{\geom}$ of $\mathbf{(S+)}$-hypergeometric sheaves, must be either almost quasisimple or an extraspecial normalizer as defined below. 

\begin{dfn}
An \emph{extraspecial normalizer} in characteristic $p$ is a finite group $G< \GL(V)$ which contains a normal subgroup $R=\mathbf{Z}(R)E$, where $V=\mathbb{C}^{p^n}$ for some $1\leq n\in \mathbb{Z}$, $E$ is an extraspecial $p$-group of order $p^n$ which acts irreducibly on $V$, and either $R=E$ or $\mathbf{Z}(R)\cong C_4$, the cyclic group of order $4$.
\end{dfn}

In \cite{KT2, KT4, KRLT4}, Katz, Rojas-Le\'on and Tiep determined whether each triple $(G,V,g)$ of a finite group $G$ which is either almost quasisimple or an extraspecial normalizer, its faithful complex representation $V$, and an element $g\in G$ which acts as an $\mathsf{ssp}$-element on $V$ can be realized as the $G_{\geom}$ and $g_0$ of an irreducible hypergeometric sheaf. The only case which was not completely covered by their results is the case of extraspecial normalizers in characteristic $2$.

\section{$\mathsf{m2sp}$-elements of extraspecial normalizers}
To find the candidates for $g_\infty$, we first find all $\mathsf{m2sp}$-elements of odd order in extraspecial normalizers in characteristic $2$. This is just an extension of \cite[Theorem 8.5]{KT2}, which classifies the $\mathsf{ssp}$-elements.

\begin{thm}\label{m2sp}
Let $V=\mathbb{C}^{2^n}$ and let $G$ be a finite irreducible subgroup of $\GL(V)$ that satisfies $(\mathbf{S}+)$ and is an extraspecial normalizer, so that $G\rhd R=\mathbf{Z}(R)E$ for some normal subgroup $R$ which contains an extraspecial $2$-group $E=2^{1+2n}_{\epsilon}$ that acts irreducibly on $V$, and either $R=E$ or $\mathbf{Z}(R)\cong C_4$. Suppose that an element $g\in G$ of odd order has no eigenvalues with multiplicity larger than $2$ on $V$ and that $2^n\geq 16$. Then there exists integers $a_1,\dots, a_t$ such that $a_1+\cdots+a_t = n$, integers $r_1,\dots, r_t$ such that $0\leq r_i< 2^{a_i}-\epsilon_i$ where $\epsilon_i=-1$ for $2\leq i\leq t$ and $\epsilon_1=(-1)^{t-1}\epsilon$, so that the image of $g$ in $G/\mathbf{Z}(G)R$ lies in a maximal torus $$C_{2^{a_1}-\epsilon_1}\times\cdots\times C_{2^{a_t}-\epsilon_t}<\mathbf{O}_{2a_1}^{\epsilon_1}\times\cdots\times \mathbf{O}_{2a_t}^{\epsilon_t}\leq \mathbf{O}_{2n}^{\epsilon}.$$
Moreover, one of the following holds: 
\begin{enumerate}[label={\upshape{(\alph*)}}]
\item $2^{a_1}+1,\dots,2^{a_t}+1$ are pairwise coprime, and $\gcd(r_i, 2^{a_i}+1)=1$ for all $i$. In this case $\overline{\mathsf{o}}(g)= \prod_{i=1}^{t}(2^{a_i}+1)$. (This is precisely the situation of \cite[Theorem 8.5(ii)]{KT2}.)
\item $2^{a_1}+1,\dots,2^{a_{t-1}}+1$ are pairwise coprime, $\gcd(r_i, 2^{a_i}+1)=1$ for $1\leq i\leq t-1$, $a_t=1$, and $r_t=0$. In this case $\overline{\mathsf{o}}(g) = \prod_{i=1}^{t-1}(2^{a_i}+1)=(\prod_{i=1}^{t}(2^{a_i}+1))/3$.
\item $2^{a_1}+1,\dots,2^{a_{t-1}}+1$ are pairwise coprime, $\gcd(r_i, 2^{a_i}+1)=1$ for all $i$, $a_t=1$, and $a_i$ is odd for some (unique) $i<t$. In this case, $n$ is even and $\overline{\mathsf{o}}(g) = \prod_{i=1}^{t-1}(2^{a_i}+1)=(\prod_{i=1}^{t}(2^{a_i}+1))/3$. 
\item $2^{a_1}-1, 2^{a_2}+1,\dots, 2^{a_t}+1$ are pairwise coprime, $\gcd(r_i, 2^{a_i}+1)=1$ for $2\leq i\leq t$, and $\gcd(r_1,2^{a_1}-1)=1$.
\end{enumerate}
In cases (a), (b) and (c), the spectrum of $g$ on $V$ is $$\left\lbrace\xi \prod_{i=1}^{t}\zeta_i^{r_i} \mid \zeta_i\in \mu_{2^{a_i}+1}\setminus\{1\}\right\rbrace$$ as a multiset, where $\xi$ is a root of unity of odd order. In case (d), the spectrum of $g$ on $V$ is $$\left\lbrace\xi\prod_{i=1}^{t}\zeta_i^{r_i} \mid \zeta_1\in \mu_{2^{a_1}-1}\sqcup\{1\}, \zeta_i\in \mu_{2^{a_i}+1}\setminus\{1\}\text{ for }i\geq 2\right\rbrace$$ as a multiset, where $\xi$ is again a root of unity of odd order.
\end{thm}

\begin{proof}
We use the notations in the proof of \cite[Theorem 8.5]{KT2}. By that proof we know that $V=V_1\otimes V_2\otimes \cdots \otimes V_t$ and $E=E_1\circ E_2\circ \cdots \circ E_t$ where $V_i=\mathbb{C}^{2^{a_i}}$ and $E_i\cong 2_{\epsilon_i}^{1+2a_i}$ for $\epsilon_i=\pm$ ($\epsilon_i=\pm 1$ when interpreted as a number). Also, $g=zh$ for some $z\in \mathbf{Z}(\GL(V))$ and $h\in \langle s_1, \dots, s_t\rangle$, where $s_i\in \Sp^{\epsilon_i}(V_i)<\Sp^\epsilon(V)$ is an element of order $2^{a_i}-\epsilon_i$ whose spectrum on $V_i$ is $\mu_{2^{a_i}+1}\setminus \{1\}$ if $\epsilon_i=-$, and $\mu_{2^{a_i}-1}$ (with $1$ having multiplicity $2$) if $\epsilon_i=+$. We may write $h=s_1^{r_1}s_2^{r_2}\cdots s_t^{r_t}$ for $0\leq r_i< 2^{a_i}-\epsilon_i$.

Suppose that $g$ has no eigenvalues of multiplicity larger than $2$ on $V$. Since $G$ acts irreducibly on $V$, $z$ acts as a scalar. Hence $h=gz^{-1}$, whose spectrum is given by the product of the spectra of $s_i^{r_i}$'s as described above, also has no eigenvalue of multiplicity larger than $2$. If $\epsilon_{i_1}=\epsilon_{i_2}=+$ for some $i_1\neq i_2$, then the spectrum of $s_{i_1}s_{i_2}$ on $V_{i_1}\otimes V_{i_2}$ has eigenvalue $1$ with multiplicity $4$, whence the spectrum of $h$ must also have eigenvalue with multiplicity at least $4$. Therefore, $\epsilon_{i}$ can be $+$ for at most one $i$. 

Suppose that $\epsilon_i=-$ for every $i$. Then the spectrum of $h$ on $V$ is $$\left\lbrace\prod_{i=1}^{t}\zeta_i^{r_i} \mid \zeta_i\in \mu_{2^{a_i}+1}\setminus\{1\}\right\rbrace\text{ (as a multiset)}.$$ In order to have no eigenvalue of multiplicity larger than $2$, we must have $\gcd(r_i, 2^{a_i}+1)=1$ for all $i$ with at most one possible exception $i=i_0$ for which $(r_{i_0},a_{i_0})=(0,1)$. Also, if $d:=\gcd(2^{a_{i_1}}+1,2^{a_{i_2}}+1)>1$ for some $i_1\neq  i_2$ and $\gcd(r_{i_1},2^{a_{i_1}}+1) = \gcd(r_{i_2},2^{a_{i_2}}+1)=1$, then we may choose $(\zeta_i,\zeta_j)= (\zeta,\zeta^{-1}), (\zeta^2,\zeta^{-2}),\dots,(\zeta^{-1},\zeta)$ for some primitive $d$th root of unity $\zeta$ to get $d-1$ copies of some eigenvalue. The only possibility is $d=3$. Moreover, if both $2^{a_{i_1}}+1$ and $2^{a_{i_2}}+1$ are larger than $3$, then we can instead choose $(\zeta_{i_1},\zeta_{i_2}) = (\xi,\eta), (\xi\zeta,\eta\zeta^{-1}),(\xi\zeta^{-1}, \eta\zeta)$ for some primitive $2^{a_{i_1}}+1$th and $2^{a_{i_2}}+1$th roots of unity $\xi$ and $\eta$ to get $3$ copies of same eigenvalue. Therefore, either $a_i=1$ or $a_j=1$. Of course, at most one such $i_0$ and at most one such pair $(i_1,i_2)$ can exist, and if they both exist then $i_0\in \{i_1,i_2\}$. After reindexing if necessary, we can assume that $a_1>a_2>\cdots > a_{t-1}\geq a_t$, and we are in one of the situations (a), (b) and (c).

Next, suppose that $\epsilon_1=+$ and $\epsilon_2=\cdots=\epsilon_t=-$. The spectrum of $h$ on $V$ is now $$\left\lbrace\prod_{i=1}^{t}\zeta_i^{r_i} \mid \zeta_1\in \mu_{2^{a_1}-1}\sqcup\{1\}, \zeta_i\in \mu_{2^{a_i}+1}\setminus\{1\}\text{ for }i\geq 2\right\rbrace\text{ (as a multiset)}.$$ Since $1$ has multiplicity $2$ in $\mu_{2^{a_1}-1}\sqcup\{1\}$, the numbers $\prod_{i=2}^t\zeta_i^{r_i}$ should be all distinct, and no two of them differ by a factor of a $(2^{a_1}-1)$th root of unity. Therefore we are in the situation (d). 
\end{proof}

In the above theorem, we had the condition that $2^{a_i}+1$ are pairwise coprime. It might be convenient to note the following easy equivalent condition.
\begin{lem}\label{2powerplus1}
$2^{a}+1$ and $2^{b}+1$ are coprime if and only if $a$ and $b$ have different $2$-parts (the highest power of $2$ dividing it). 
\end{lem}

\begin{proof}
Let $a_2$ and $b_2$ be the $2$-parts of $a$ and $b$. Note that $2^{a}+1 = (2^{a_2}+1)(2^{a-a_2}-2^{a-2a_2}+\cdots -2^{a_2}+1)$. The ``only if'' direction is immediate from this. For the converse, assume that $a_2>b_2$, and observe that $$\gcd(2^{a}+1,2^b+1)  \mid \gcd(2^{2a}-1, 2^{2b}-1) = 2^{2\gcd(a,b)}-1= (2^{\gcd(a,b)}-1)(2^{\gcd(a,b)}+1).$$ Since $2^{\gcd(a,b)}-1$ divides $2^a-1$, it is coprime to $2^a+1$. Also, $\gcd(a,b)$ is not divisible by $a_2$, so $a/\gcd(a,b)$ is even. Hence $2^{\gcd(a,b)}+1$ also divides $2^a-1$, so it is coprime to $2^a+1$. Therefore $\gcd(2^a+1,2^b+1)=1$. 
\end{proof}

\section{Local monodromy at $\infty$ of hypergeometric sheaves}

Let $\mathcal{H}$ be an irreducible hypergeometric sheaf over a field $\mathbb{F}_q$ of characteristic $2$ which is of type $(D,M)$ with $W:=D-M>0$ and whose geometric monodromy group $G:=G_{\geom}$ is an extraspecial normalizer of characteristic $2$. By \cite[Theorem 9.19(ii)]{KT2}, $G$ has a normal extraspecial $2$-subgroup $E\cong 2^{1+2n}_{\pm}$ where $2^n=D$, and $\mathbf{Z}(E)$ is the Sylow $2$-subgroup of $\mathbf{Z}(G)$; in other words, $R=E$. 

In this section, we determine which of the elements in \autoref{m2sp} can appear as $g_\infty$ of $\mathcal{H}$. We must take into account the structure of the image $Q$ of $P(\infty)$. 

\begin{prop}\label{3-possibilities}
At least one of the following holds:
\begin{enumerate}[label={\upshape{(\alph*)}}]
\item $\mathcal{H}$ is Kloosterman,
\item $Q\cap E=1$, so the image of $Q$ in $G/\mathbf{Z}(G)E$ is isomorphic to $Q$, or
\item $(Q\cap E)\mathbf{Z}(E)$ is elementary abelian, each irreducible constituent of $\Wild|_{Q\cap E}$ is nontrivial, and $(\dim \Tame)\cdot |\{Q\cap Q^e\cap E\mid e\in E\}|=D=2^n$.
\end{enumerate}

\end{prop}

\begin{proof}
Suppose that $\mathcal{H}$ is not in the cases (a) and (b), so that $Q\cap E$ is nontrivial and $\dim \Tame>0$. By \cite[Proposition 4.8]{KT2}, $Q\cap \mathbf{Z}(G)=1$ and $|\mathbf{Z}(G)|_2\leq 2$. Since $Q$ is a $2$-group and $\mathbf{Z}(E)\leq \mathbf{Z}(G)$, we have $Q\cap\mathbf{Z}(E)=1$. Therefore, $Q\cap E$ is isomorphic to its image in $E/\mathbf{Z}(E)$, so $(Q\cap E)\mathbf{Z}(E)$ is an elementary abelian $2$-group. 

Let $\Phi$ be the monodromy representation of $\mathcal{H}$. Consider its restrictions to the subgroups $Q\cap E \lhd (Q\cap E)\mathbf{Z}(E)\unlhd E$. Since $Q\cap E\unlhd Q\unlhd J$, the restriction $\Phi|_{Q\cap E}$ decomposes into $\Tame|_{Q\cap E}\oplus \Wild|_{Q\cap E} = (\dim \Tame)\cdot \mathds{1}_{Q\cap E}\oplus \Wild|_{Q\cap E}$. Hence, there exists some irreducible constituent $\Psi$ of $\Phi|_{(Q\cap E)\mathbf{Z}(E)}$ (which is one-dimensional since $(Q\cap E)\mathbf{Z}(E)$ is abelian as we saw above) whose restriction to $Q\cap E$ is trivial. On the other hand, since $\Phi|_E$ is irreducible and faithful, by Clifford theory, no irreducible constituent of its restriction to the normal subgroup $(Q\cap E)\mathbf{Z}(E)$ can contain a normal subgroup of $E$ in the kernel, so in particular $\Psi\neq \mathds{1}_{(Q\cap E)\mathbf{Z}(E)}$. Therefore $\ker\Psi = Q\cap E$, and each irreducible constituent of $\Phi|_{(Q\cap E)\mathbf{Z}(E)}$ has kernel $Q^e\cap E$ for some $e\in E$. In particular, no irreducible constituent of $\Phi|_{(Q\cap E)\mathbf{Z}(E)}$ other than $\Psi$ restricts to $\mathds{1}_{Q\cap E}$, so the multiplicity of $\Psi$ in $\Phi|_{(Q\cap E)\mathbf{Z}(E)}$ is exactly the multiplicity of $\mathds{1}_{Q\cap E}$ in $\Phi|_{Q\cap E}$, and it divides $\dim \Phi$.

Now consider the restriction of the irreducible representation $\Wild$ of $J$ to the normal subgroups $Q\cap E\unlhd Q\unlhd J$. If $\Wild|_{Q\cap E}$ has $\mathds{1}_{Q\cap E}$ as an irreducible constituent, then there exists an irreducible constituent $\Theta$ of $\Wild|_Q$ with $Q\cap E\leq \ker\Theta$. Since $Q\cap E\unlhd J$, it follows that $Q\cap E\leq \bigcap_{j\in J}(\ker\Theta)^j= \ker \Wild|_Q$. This is impossible since $\Phi$ is faithful and $Q\leq \ker\Tame$. Therefore, $\Wild|_{Q\cap E}$ cannot have $\mathds{1}_{Q\cap E}$ as an irreducible constituent, so the multiplicity of $\mathds{1}_{Q\cap E}$ in $\Phi|_{Q\cap E}$ is exactly $\dim \Tame$. By the previous paragraph, this number divides $\dim \Phi=D=2^n$.
\end{proof}

As we now know what are the possible intersections $Q\cap E$, it is natural to ask what are the possible quotients $Q/(Q\cap E)$. As $g_\infty$ normalizes $Q$, we can simply find all $2$-subgroups of $G/\mathbf{Z}(G)E$ which are normalized by the elements described in \autoref{m2sp}.

\begin{prop}\label{Q}
Let $a_1,\dots, a_t$ be positive integers such that $a_1+\cdots+a_t=n$, and let $g=(g_1,\dots,g_t)$ be an element of $\Sp_{2a_1}(2)\times \cdots \times \Sp_{2a_t}(2)\leq \Sp_{2n}(2)$. Suppose that $P$ is a $2$-subgroup of $\Sp_{2n}(2)$ such that $g\in \mathbf{N}_{\Sp_{2n}(2)}(P)$.
\begin{enumerate}[label={\upshape{(\alph*)}}]
\item Suppose that $g_1,\dots,g_t$ are irreducible over $\mathbb{F}_2$, have distinct orders $o(g_i)$, and every nonzero orbit of the action of $g_i$ on $\mathbb{F}_2^{2a_i}$ has length $o(g_i)$. Then $P$ must be trivial.
\item Suppose that $g_1,\dots,g_{t-1}$ are irreducible over $\mathbb{F}_2$, have distinct orders $o(g_i)$, and every nonzero orbit of the action of $g_i$ on $\mathbb{F}_2^{2a_i}$ has length $o(g_i)$. Also suppose that $V_t=V_{t_1}\oplus V_{t_2}$ for some $a_t$-dimensional subspaces $V_{t_1}$ and $V_{t_2}$, $g_{t}=g_{t,1}\oplus g_{t,2}\in (\GL(V_{t_1})\times\GL(V_{t_2}))\cap \Sp_{2a_t}(2)$ for some $g_{t,1}$ and $g_{t,2}$ which are irreducible over $\mathbb{F}_2$, has same order $o(g_t)$ which is not equal to any other $o(g_i)$, and every nonzero orbit of the actions of $g_{t,1}$ and $g_{t,2}$ on $\mathbb{F}_2^{a_t}$ has length $o(g_t)$. Then $P$ must be elementary abelian of order not exceeding $2^{a_t^2}$.
\end{enumerate}

\end{prop}

\begin{proof}
Let $A_m$ be the $m\times m$ matrix $\begin{pmatrix}
& & 1 \\ & \reflectbox{$\ddots$} & \\ 1 & &
\end{pmatrix}$. Let $V=V_1\oplus \cdots \oplus V_t$ be the natural module of $\Sp_{2a_1}(2)\times \cdots \times \Sp_{2a_t}(2)$ with the symplectic form $A_{2a_1}\oplus\cdots\oplus A_{2a_t}$. 

(a) We first show that $V_1,\dots,V_t$ are the minimal $\langle g\rangle$-submodules of $V$. Since each $g_i$ is irreducible, it is clear that each $V_i$ is a minimal $\langle g\rangle$-submodule of $V$. Conversely, suppose that a minimal submodule $U$ contains a vector $v=(v_1,\dots,v_t)\in V_1\oplus \cdots \oplus V_t$ with at least two nonzero components $v_i, v_j$. We may assume that $o(g_i)<o(g_j)$. Then $g^{o(g_i)}v-v = (g_1^{o(g_i)}v_1-v_1, \dots, g_t^{o(g_t)}v_t-v_t)\in M$. The $i$th component of this vector is $g_i^{o(g_i)}v_i-v_i=0$, while the $j$th component is nonzero. We may repeat this until we get a vector with only one nonzero component. Therefore, $M$ intersects at least one of $V_1,\dots, V_t$ nontrivially, so by the minimality, $M$ is one of them.

Since $g_1,\dots,g_t$ are irreducible over $\mathbb{F}_2$ and have pairwise coprime orders, the minimal $\langle g\rangle$-submodules of $V$ are $V_1,\dots,V_t$. Since $P$ is a $2$-subgroup of $\GL_{2n}(2)$ normalized by $g$, $\mathbf{C}_V(P)$ is a nontrivial $\langle g\rangle$-submodule of $V$, so it must contain some $V_i$. After relabeling if necessary, we may assume that $V_1\subseteq \mathbf{C}_V(P)$. Now the action of $P$ on $V/V_1$ fixes a nontrivial $\langle g\rangle$-submodule by the same reason. Repeating this, we can see that after relabeling, $$P\leq \left\lbrace\left.\begin{pmatrix}
I_{2a_1} & X_{1,2} & \cdots &X_{1,t} \\ 0 & I_{2a_2} & \cdots & X_{2,t} \\0 &0 & \ddots & \vdots \\ 0&0&0&I_{2a_t}
\end{pmatrix} \right| X_{i, j}\in M_{2a_i\times 2a_j}(\mathbb{F}_2) \right\rbrace.$$ On the other hand, the elements of $P$ preserves the symplectic form $A_{2a_1}\oplus\cdots\oplus A_{2a_t}$, so 
\begin{align*}
&\begin{pmatrix}
A_{2a_1} &  & & \\  & A_{2a_2} &  &  \\ & & \ddots &  \\ &&&A_{2a_t}
\end{pmatrix}\\=&\begin{pmatrix}
I_{2a_1} &  &  & \\ X_{1,2}^T & I_{2a_2} &  &  \\\vdots &\vdots & \ddots &  \\ X_{1,t}^T&X_{2,t}^T&\cdots&I_{2a_t}
\end{pmatrix}\begin{pmatrix}
A_{2a_1} &  & & \\  & A_{2a_2} &  &  \\ & & \ddots &  \\ &&&A_{2a_t}
\end{pmatrix}\begin{pmatrix}
I_{2a_1} & X_{1,2} & \cdots &X_{1,t} \\ 0 & I_{2a_2} & \cdots & X_{2,t} \\0 &0 & \ddots & \vdots \\ 0&0&0&I_{2a_t}
\end{pmatrix}\\=& \begin{pmatrix}
A_{2a_1} & & &\\X_{1,2}^TA_{2a_1} & A_{2a_2} & & \\ \vdots & \vdots & \ddots & \\ X_{1,t}^TA_{2a_1} & X_{2,t}^TA_{2a_2} & \cdots & A_{2a_t}
\end{pmatrix}\begin{pmatrix}
I_{2a_1} & X_{1,2} & \cdots &X_{1,t} \\ 0 & I_{2a_2} & \cdots & X_{2,t} \\0 &0 & \ddots & \vdots \\ 0&0&0&I_{2a_t}
\end{pmatrix}\\=&\begin{pmatrix}
A_{2a_1} & A_{2a_1}X_{1,2} & \cdots & A_{2a_1}X_{1,t} \\ X_{1,2}^TA_{2a_1} & X_{1,2}^TA_{2a_1}X_{1,2} + A_{2a_2} & \cdots & X_{1,2}^TA_{2a_1}X_{1,t} + A_{2a_1}X_{2,t} \\ \vdots & \vdots & \ddots & \vdots \\ X_{1,t}^TA_{2a_1} & X_{1,t}^TA_{2a_1}X_{1,2} + X_{2,t}^TA_{2a_2} & \cdots & X_{1,t}^TA_{2a_1}X_{1,t} + \cdots + A_{2a_t}
\end{pmatrix}
\end{align*}
which shows that $X_{i,j}=0$ for all pairs $(i,j)$. Therefore $P$ is trivial.

(b) By the same argument as in (a), we see that the minimal $g_\infty$-invariant subspaces of $V$ are $V_1,\dots,V_{t-1}$ and certain subspaces of $V_t$. Moreover, the minimal polynomials of $g_{t,1}$ and $g_{t,2}$ are both irreducible of degree $a_t$, so their least common multiple is the minimal polynomial of $g_t$. Hence, every minimal $g_\infty$-invariant subspace of $V_t$ has dimension $a_t$. Therefore, for a suitable choice of basis, we have $$P\leq \left\lbrace\begin{pmatrix}
I_{2a_1} & X_{1,2} & \cdots &X_{1,t-1} & X_{1,t} & X_{1,t+1} &\\ 0 & I_{2a_2} & \cdots & X_{2,t-1} & X_{2,t} & X_{2,t+1} \\\vdots &\vdots & \ddots & \vdots & \vdots & \vdots \\ 0&0&\cdots&I_{2a_{t-1}} & X_{t-1,t} &X_{t-1,t+1} \\ X_{t,1} & X_{t,2} & \cdots & X_{t,t-1} & I_{a_t} & X_{t,t+1}\\ X_{t+1,1} & X_{t+1,2} & \cdots & X_{t+1,t-1} &  0 & I_{a_t}
\end{pmatrix} \right\rbrace$$ where the symplectic form in this basis becomes $A_{2a_1}\oplus\cdots\oplus A_{2a_{t-1}}\oplus B$ for some symplectic form $B =\begin{pmatrix}
B_{1} & B_{2} \\ B_{2}^T & B_{3}
\end{pmatrix}$ on $\mathbb{F}_2^{2a_t}$, where $B_1, B_2, B_3$ are $a_t\times a_t$ matrices. Also, there exists $0\leq i_1\leq i_2\leq t-1$ such that $X_{i,t}=0$ for all $i> i_1$, $X_{i,t+1}=0$ for all $i> i_2$, $X_{t,i}=0$ for all $i\leq i_1$, and $X_{t+1,i}=0$ for all $i\leq i_2$. In particular, for each $i=1,\dots,t-1$, at least one of $X_{i,t}$ and $X_{t,i}$ is $0$, and at least one of $X_{i,t+1}$ and $X_{t+1,i}$ is $0$.

Since each element of $P$ preserves the symplectic form, we have the following equality:
{\scriptsize
\begin{align*}
&\begin{pmatrix}
A_{2a_1} &  & && &\\  & A_{2a_2} &  & && \\ & & \ddots & & &\\ &&&A_{2a_{t-1}}&&\\&&&&B_1&B_2\\&&&&B_2^T&B_3
\end{pmatrix}\\=&\left(\begin{smallmatrix}
I_{2a_1} &  &  & & X_{t,1}^T & X_{t+1,1}^T\\ X_{1,2}^T & I_{2a_2} &  &  & X_{t,2}^T & X_{t+1,2}^T\\\vdots &\vdots & \ddots & & \vdots & \vdots\\ X_{1,t-1}^T&X_{2,t-1}^T&\cdots&I_{2a_{t-1}}&X_{t,t-1}^T & X_{t+1,t-1}^T\\ X_{1,t}^T & X_{2,t}^T & \cdots & X_{t-1,t}^T & I_{a_t} & 0 \\ X_{1,t+1}^T & X_{2,t+1}^T & \cdots & X_{t-1,t+1}^T & X_{t,t+1}^T & I_{a_t}
\end{smallmatrix}\right)\left(\begin{smallmatrix}
A_{2a_1} &  & && &\\  & A_{2a_2} &  & && \\ & & \ddots & & &\\ &&&A_{2a_{t-1}}&&\\&&&&B_1&B_2\\&&&&B_2^T&B_3
\end{smallmatrix}\right)\left(\begin{smallmatrix}
I_{2a_1} & X_{1,2} & \cdots &X_{1,t-1} & X_{1,t} & X_{1,t+1} &\\ 0 & I_{2a_2} & \cdots & X_{2,t-1} & X_{2,t} & X_{2,t+1} \\\vdots &\vdots & \ddots & \vdots & \vdots & \vdots \\ 0&0&\cdots&I_{2a_{t-1}} & X_{t-1,t} &X_{t-1,t+1} \\ X_{t,1} & X_{t,2} & \cdots & X_{t,t-1} & I_{a_t} & X_{t,t+1}\\ X_{t+1,1} & X_{t+1,2} & \cdots & X_{t+1,t-1} &  0 & I_{a_t}
\end{smallmatrix}\right)\\=& \left(\begin{smallmatrix}
A_{2a_1} & 0&\cdots &0 & X_{t,1}^TB_1+X_{t+1,1}^TB_2^T & X_{t,1}^TB_2+X_{t+1,1}^TB_3\\X_{1,2}^TA_{2a_1} & A_{2a_2} & \cdots &0&X_{t,2}^TB_1 + X_{t+1,2}^TB_2^T & X_{t,2}^TB_2 + X_{t+1,2}^TB_3 \\ \vdots & \vdots & \ddots & \vdots & \vdots & \vdots\\ X_{1,{t-1}}^TA_{2a_1} & X_{2,{t-1}}^TA_{2a_2} & \cdots & A_{2a_{t-1}} & X_{t,t-1}^TB_1 + X_{t+1,t-1}^TB_2^T & X_{t,t-1}^TB_2+X_{t+1,t-1}^TB_3 \\ X_{1,t}^TA_{2a_1} & X_{2,t}^TA_{2a_2} & \cdots & X_{t-1,t}^TA_{2a_{t-1}} & B_1 & B_2 \\ X_{1,t+1}^TA_{2a_1} & X_{2,t+1}^TA_{2a_2} & \cdots & X_{t-1,t+1}^TA_{2a_{t-1}} & X_{t,t+1}^TB_1 + B_2^T & X_{t,t+1}^TB_2 + B_3
\end{smallmatrix}\right)\left(\begin{smallmatrix}
I_{2a_1} & X_{1,2} & \cdots &X_{1,t-1} & X_{1,t} & X_{1,t+1} &\\ 0 & I_{2a_2} & \cdots & X_{2,t-1} & X_{2,t} & X_{2,t+1} \\\vdots &\vdots & \ddots & \vdots & \vdots & \vdots \\ 0&0&\cdots&I_{2a_{t-1}} & X_{t-1,t} &X_{t-1,t+1} \\ X_{t,1} & X_{t,2} & \cdots & X_{t,t-1} & I_{a_t} & X_{t,t+1}\\ X_{t+1,1} & X_{t+1,2} & \cdots & X_{t+1,t-1} &  0 & I_{a_t}
\end{smallmatrix}\right)
\end{align*}
}
If $i_2=0$, then $X_{i,t}=X_{i,t+1}=0$ for all $i=1,\dots,t-1$. If $i_1>0$, then $X_{t,1}=X_{t+1,1}=0$, so by computing the first row of the product above, we get $X_{1,2}=\cdots=X_{1,t+1}=0$. If $i_1=0<i_2$, then $X_{i,t}=X_{t+1,1}=0$ for $i=1,\dots,t-1$, so by computing the $(1,t+1)$- and $(t,t+1)$-entries, we get $A_{2a_1}X_{1,t+1} + X_{t,1}^TB_1X_{t,t+1} + X_{t,1}^TB_2=0$ and $B_1X_{t,t+1}+B_2=0$, so that $X_{1,t+1}=0$. Therefore, regardless of the numbers $i_1$ and $i_2$ and the symplectic form $B$, we always have $X_{1,t}=X_{1,t+1}=0$. 

Similarly, assume that $X_{j,t}=X_{j,t+1}=0$ for all $j<i$. If $i_2< i$, then $X_{i,t}=X_{i,t+1}=0$. If $i_1\geq i$, then $X_{t,j}=X_{t+1,j}=0$ for $j=1,\dots,i$, so by computing the $(i,t)$- and $(i,t+1)$-entry, we get $A_{2a_i}X_{i,t}=0$ and $A_{2a_i}X_{i,t+1}=0$. If $i_1<i\leq i_2$, then $X_{k,t}=X_{t+1,j}=0$ for $k=i,\dots,t-1$ and $j=1,\dots, i$, so by computing the $(i,t+1)$ and $(t,t+1)$-entries, we get $A_{2a_i}X_{i,t+1} + X_{t,i}^TB_1X_{t,t+1}+ X_{t,i}^TB_2 =0$ and $B_1X_{t,t+1}+B_2=0$, so that $A_{2a_i}X_{i,t+1}=0$. Therefore, in any case, we get $X_{i,t}=X_{i,t+1}=0$. By induction we get $X_{i,t}=X_{i,t+1}=0$ for all $i\in \{1,\dots,t-1\}$. The above matrix now becomes 
\begin{align*}
\left(\begin{smallmatrix}
A_{2a_1} & 0&\cdots &0 & X_{t,1}^TB_1+X_{t+1,1}^TB_2^T & X_{t,1}^TB_2+X_{t+1,1}^TB_3\\X_{1,2}^TA_{2a_1} & A_{2a_2} & \cdots &0&X_{t,2}^TB_1 + X_{t+1,2}^TB_2^T & X_{t,2}^TB_2 + X_{t+1,2}^TB_3 \\ \vdots & \vdots & \ddots & \vdots & \vdots & \vdots\\ X_{1,{t-1}}^TA_{2a_1} & X_{2,{t-1}}^TA_{2a_2} & \cdots & A_{2a_{t-1}} & X_{t,t-1}^TB_1 + X_{t+1,t-1}^TB_2^T & X_{t,t-1}^TB_2+X_{t+1,t-1}^TB_3 \\ 0 & 0 & \cdots & 0 & B_1 & B_2 \\0 & 0 & \cdots &0 & X_{t,t+1}^TB_1 + B_2^T & X_{t,t+1}^TB_2 + B_3
\end{smallmatrix}\right)\left(\begin{smallmatrix}
I_{2a_1} & X_{1,2} & \cdots &X_{1,t-1} & 0 & 0 &\\ 0 & I_{2a_2} & \cdots & X_{2,t-1} & 0 & 0 \\\vdots &\vdots & \ddots & \vdots & \vdots & \vdots \\ 0&0&\cdots&I_{2a_{t-1}} & 0 &0 \\ X_{t,1} & X_{t,2} & \cdots & X_{t,t-1} & I_{a_t} & X_{t,t+1}\\ X_{t+1,1} & X_{t+1,2} & \cdots & X_{t+1,t-1} &  0 & I_{a_t}
\end{smallmatrix}\right)
\end{align*}
By computing the $t$th column, we immediately see that $X_{t,i}^TB_1+X_{t+1,i}^TB_2^T=0$ for $i=1,\dots,t-1$ and $X_{t,t+1}^TB_1=0$. Then we can also compute the $t+1$th column and see that $X_{t,i}^TB_2+X_{t+1,i}^TB_3=0$ for $i=1,\dots,t-1$ and $B_2^TX_{t,t+1}+X_{t,t+1}^TB_2=0$. The matrix is now
\begin{align*}
\left(\begin{smallmatrix}
A_{2a_1} & 0&\cdots &0 & 0& 0\\X_{1,2}^TA_{2a_1} & A_{2a_2} & \cdots &0&0& 0 \\ \vdots & \vdots & \ddots & \vdots & \vdots & \vdots\\ X_{1,{t-1}}^TA_{2a_1} & X_{2,{t-1}}^TA_{2a_2} & \cdots & A_{2a_{t-1}} & 0 & 0 \\ 0 & 0 & \cdots & 0 & B_1 & B_2 \\0 & 0 & \cdots &0 &  B_2^T & X_{t,t+1}^TB_2 + B_3
\end{smallmatrix}\right)\left(\begin{smallmatrix}
I_{2a_1} & X_{1,2} & \cdots &X_{1,t-1} & 0 & 0 &\\ 0 & I_{2a_2} & \cdots & X_{2,t-1} & 0 & 0 \\\vdots &\vdots & \ddots & \vdots & \vdots & \vdots \\ 0&0&\cdots&I_{2a_{t-1}} & 0 &0 \\ X_{t,1} & X_{t,2} & \cdots & X_{t,t-1} & I_{a_t} & X_{t,t+1}\\ X_{t+1,1} & X_{t+1,2} & \cdots & X_{t+1,t-1} &  0 & I_{a_t}
\end{smallmatrix}\right)
\end{align*}
so it is now straightforward to check that $X_{i,j}=0$ for $1\leq i<j\leq t-1$. Therefore, we get 
$$P\leq \left.\left\lbrace\begin{pmatrix}
I_{2a_1} & &  & &  &  &\\  & I_{2a_2} &  &  &  &  \\&& \ddots &  &&  \\ &&&I_{2a_{t-1}} &  & \\ &  &  & & I_{a_t} & X_{t,t+1}\\ &  & &  &  & I_{a_t}
\end{pmatrix}\right| \begin{matrix}X_{t,t+1}\in M_{a_t,a_t}(2),\\ X_{t,t+1}^TB_1=0, \\B_2^TX_{t,t+1}=(B_2^TX_{t,t+1})^T\end{matrix} \right\rbrace$$ so $P$ is an elementary abelian $2$-group of order at most $2^{a_t^2}$.

\end{proof}

\begin{prop}\label{Wild_lower_bound}
$W/D\geq (2-\sqrt{2})/4>0.146$. If $W$ is even, then $W/D\geq 7(2-\sqrt{2})/16 > 0.256$. In general, $W/D> ((2-\sqrt{2})/2)(1-1/W_0)$, where $W_0$ is the $2'$-part of $W$.
\end{prop}
\begin{proof}
This is almost identical to the first part of the proof of \cite[Theorem 7.4]{KT2}. The only thing that has to be changed is that in the notation of that proof, we now have $p=r=2$ so that $|Q|\geq 2$, and $\overline{g}$ is now a $2$-element of $\Sp_{2n}(2)$. In this setting, we still get the bound $W/D\geq (1-1/\sqrt{2})(1-1/2) = (2-\sqrt{2})/4 > 0.146$. If $W$ is even, then $Q$ has an irreducible representation of even dimension, so it cannot be abelian. Therefore $|Q|\geq 8$ and we get $W/D\geq (1-1/\sqrt{2})(1-1/8)=7(1-1/\sqrt{2})/8=7(2-\sqrt{2})/16>0.256$. For the last statement, recall that $\Wild|_Q$ is a sum of $W_0$ distinct nontrivial irreducible constituents. Therefore $|Q|\geq |\Irr(Q)|> W_0$.
\end{proof}

Now we determine the local monodromy of $\mathcal{H}$ at $\infty$ and the downstairs characters.

\begin{thm}\label{g_infty}
Suppose that an irreducible hypergeometric sheaf $\mathcal{H}$ has an extraspecial normalizer as its $G_{\geom}$. Then $g_\infty$ is as in one of the cases described in \autoref{m2sp}, with the following additional conditions:
\begin{enumerate}[label={\upshape{(\alph*)}}]
\item It is as in \autoref{m2sp}(a), with $t=1$. $\mathcal{H}$ is Kloosterman in this case.
\item It is as in \autoref{m2sp}(d), $t= 2$, $2\mid na_2$, and the downstairs characters are $\Char(2^{a_2}+1)\setminus\{\mathds{1}\}$.
\item It is as in \autoref{m2sp}(d), $t=1$, and the downstairs character is $\{\mathds{1}\}$.
\item It is as in \autoref{m2sp}(b), $n=2$, and the downstairs characters are $\Char(3)\setminus\{\mathds{1}\}$.
\end{enumerate}

\end{thm}

\begin{proof}
(1) Suppose that $g$ is as in \autoref{m2sp}(a). By \autoref{Q}(a) we know that $Q/(Q\cap E)=1$, so $Q\leq E
$. First suppose that $\mathcal{H}$ is Kloosterman. Then $\dim \Wild=D=2^n$, so the spectrum of $g$ on $\Wild$ is $\eta(\mu_{2^n+1}\setminus \{1\})$ for some root of unity $\eta$ of odd order. Since there is no $\Tame$ in this case, this is the entire spectrum of $g$, so $t=1$.

Suppose now that $\mathcal{H}$ is not Kloosterman. Since $Q$ is nontrivial, we are in the situation of \autoref{3-possibilities}(c) with $Q=Q\cap E$. In particular, $Q$ is abelian. Therefore, $\dim \Tame=D-\dim\Wild$ is odd, and it divides $D=2^n$, so $\dim \Tame=1$ and $\dim \Wild=2^n-1$. The spectrum of $g$ on $\Wild$ is $\eta\mu_{2^n-1}$ for some root of unity $\eta$ of odd order. This does not happen for $g$ as in \autoref{m2sp}(a).

(2) Now consider the case (b) of \autoref{m2sp}. Here, the spectrum of $g$ on $V$ is $$\left\lbrace \xi\prod_{i=1}^{t-1}\zeta_i\mid \zeta_i\in \mu_{2^{a_i}+1}\setminus \{1\}\right\rbrace$$ with every eigenvalue having multiplicity $2$. Therefore, if an irreducible hypergeometric sheaf has such $g$ as its $g_\infty\in G_{\geom}$, then the spectrum of $g_\infty$ on $\Wild$ must be the above set, with every eigenvalue having multiplicity $1$. Hence $\dim \Wild = D/2=2^{n-1}$. On the other hand, since $Q\cap E$ is an abelian normal subgroup of $Q$ and $Q/(Q\cap E)$ is either trivial or has order $2$ by \autoref{Q}(b), the irreducible characters of $Q$ must have degree $1$ or $2$. Therefore, the $2$-part of $\dim\Wild$ is $1$ or $2$. This forces $n=1$ or $2$. The downstairs characters must be $\{\mathds{1}\}$ if $n=1$ and $\Char(3)\setminus\{\mathds{1}\}$ if $n=2$. However, if $n=1$, then the only hypergeometric sheaf which can have such $g_\infty$ together with $g_0$ as in \autoref{m2sp}(a) is $\mathcal{H}yp_\psi(\Char(3)\setminus\{\mathds{1}\}; \mathds{1})$. The $G_{\geom}$ of this sheaf has an element of order $3$ (namely $g_0$) and a nontrivial $2$-subgroup. By the proof of \cite[Theorem 9.3(i)]{KT2}, we get $G_{\geom}=S_3$. Therefore $n\neq 1$.

(3) Next, consider the case \autoref{m2sp}(c). Let $i_{odd}<t$ be the unique index such that $a_{i_{odd}}$ is odd. Then the spectrum is given by 
\begin{align*}
&\{\xi\omega^j\prod_{i=1}^{t-1}\zeta_i^{r_i} \mid \zeta_i\in \mu_{2^{a_i}+1}\setminus \{1\}, j\in\{1,2\}\} \\=& 2\{\xi\prod_{i=1}^{t-1}\zeta_i \mid \zeta_i\in \mu_{2^{a_i}+1}\setminus\{1\}\text{ for }i\neq i_{odd}, \zeta_{i_{odd}}\in \mu_{2^{a_{i_{odd}}}+1}\setminus\{\omega,\omega^2 \}\} \\&\cup \{\xi\prod_{i=1}^{t-1}\zeta_i \mid \zeta_i\in \mu_{2^{a_i}+1}\setminus\{1\}\text{ for }i\neq i_{odd}, \zeta_{i_{odd}}\in \{\omega,\omega^2 \}\}.
\end{align*} where $\xi$ is a root of unity of odd order and $\omega$ is a primitive $3$rd root of unity. Those with multiplicity $2$ must appear exactly once in both $\Tame$ and $\Wild$. 

Suppose that $a_{i_{odd}}=1$. By the proof of \autoref{Q}(b) with $B_1=A_{2}=B_3$ and $B_2=0$, we see that $Q/(Q\cap E)=1$. The spectrum of $g$ is \begin{align*}
&2\left\lbrace\xi\prod_{i=1}^{t-1}\zeta_i \mid \zeta_i\in \mu_{2^{a_i}+1}\setminus\{1\}\text{ for }i\neq i_{odd}, \zeta_{i_{odd}}=1\right\rbrace \\&\cup \left\lbrace\xi\prod_{i=1}^{t-1}\zeta_i \mid \zeta_i\in \mu_{2^{a_i}+1}\setminus\{1\}\text{ for }i\neq i_{odd}, \zeta_{i_{odd}}\in \{\omega,\omega^2 \}\right\rbrace.
\end{align*}
Since $Q\leq E$ and $Q$ is nontrivial, we are in the situation of \autoref{3-possibilities}(c) with abelian $Q$. Hence, $\dim \Tame$ is an odd number dividing $D$. Therefore $\dim\Tame=1$. Since $\dim \Tame$ is at least the number of eigenvalues of $g$ with multiplicity $2$, we must have $t=2$ and $n=a_{i_{odd}}+a_2=2$. However, in this case, $E$ is of type $+$, so $g_0$ must be as in \autoref{m2sp}(a) with even $t$, which is impossible.

Now suppose that $a_{i_{odd}}> 1$. Then from \autoref{Q}(a), we can see that $Q/(Q\cap E)$ is trivial, so that $1\neq Q\leq E$. Also, since $\mathcal{H}$ has nonzero tame part, it is not Kloosterman. Therefore, we are again in the situation of \autoref{3-possibilities}(c) with $Q=Q\cap E$. Therefore $\dim\Tame=1$, but this is impossible since $\dim\Tame$ must be at least the number of eigenvalues of $g$ with multiplicity $2$, and there are at least $|\mu_{2^{a_{i_{odd}}}+1}\setminus \{\omega,\omega^2\}|\geq 3$ of them.

(4) For \autoref{m2sp}(d), the eigenvalues of multiplicity $2$ are $$\{\xi\zeta_2\zeta_3\cdots \zeta_{t} \mid \zeta_i\in\mu_{2^{a_i}+1}\setminus\{1\}\}\text{ if }t\geq 2\text{, and }\{\xi\}\text{ if }t=1.$$ Therefore, the spectra on $\Tame$ and $\Wild$ both contain them, so $\mathcal{H}$ is not Kloosterman. Also, \autoref{Q}(b) shows that $Q/(Q\cap E)$ is elementary abelian of order not exceeding $2^{a_1^2}$. 

If we are in the situation of \autoref{3-possibilities}(b), then $Q$ is elementary abelian and $\dim \Wild$ is odd. Hence, the eigenvalues of $g$ on $\Wild$ are the $\dim\Wild$th roots of $(\xi\zeta_2\zeta_3\cdots \zeta_t)^{\dim \Wild}$ for $\zeta_i\in \mu_{2^{a_i}+1}$ if $t\geq 2$. Hence $\prod_{i=2}^{t}(2^{a_i}+1)$ divides $\dim \Wild$, but not all $\prod_{i=2}^{t}(2^{a_i}+1)$th roots of $(\xi\zeta_2\zeta_3\cdots \zeta_t)^{\dim \Wild}$ are in the  spectrum of $g$. Therefore $t=1$, $a_1=n$ and $\dim \Wild$ divides $2^{n}-1$. By \autoref{Wild_lower_bound}, $(2^{n}-1)/\dim\Wild\in \{1,3\}$.

Suppose that $\dim\Wild=(2^n-1)/3$. Let $\chi$ be the character of $\mathcal{H}$ as a representation of $G$, and $\chi_{\Tame}=(\dim\Tame) 1_P$ and $\chi_{\Wild}$ be that of $\Tame$ and $\Wild$ as a representation of $Q$, so that $\chi|_Q=\chi_{\Tame}+\chi_{\Wild}$. Also let $x\in Q$ be any nonidentity element. Then $x$ has order $2$, so the eigenvalues of the action of $x$ on $\mathcal{H}$ are $1$ or $-1$. Therefore, $$\mathbb{Z}\ni\chi(x) = \chi_{\Tame}(x)+\chi_{\Wild}(x) \geq \dim\Tame-\dim\Wild=1+(2^n-1)/3>2^{n-2}.$$ On the other hand, by \cite[Lemma 2.4]{GT}, $|\chi(x)|$ is a power of $\sqrt{2}$, and since $\chi$ is faithful, $\chi(x)<2^n$. Therefore $\chi(x)=2^{n-1}$ for all $x\in Q\setminus \{1\}$. But then $$2^n - \frac{2^n-1}{3}=\dim\Tame= (\chi|_Q, 1_Q) = \frac{1}{|Q|}\sum_{x\in Q}\chi(x) = \frac{2^n + (|Q|-1)2^{n-1}}{|Q|} = 2^{n-1} + \frac{2^{n-1}}{|Q|}$$ so that $$|Q|  = \frac{2^{n-1}}{2^{n-1} - \frac{2^n-1}{3}} = \frac{2^{n-1}}{\frac{ 2^{n-1}+1}{3}} .$$ This forces $(2^{n-1}+1)/3 = 1$, so $n=2, t=2$, and $a_1=a_2=1$. Hence $E$ must be of type $-$, and by \autoref{m2sp}(a), $g_0$ must have order $5$. Therefore, the hypergeometric sheaf must be $\mathcal{H}yp_\psi(\Char(5)\setminus\{\mathds{1}\}; \Char(3))$ up to tensoring by a Kummer sheaf. However, this sheaf has $G_{\geom}=\mathsf{S_5}$ as studied in the proof of \cite[Theorem 9.3]{KT2}.

Suppose now that \autoref{3-possibilities}(c) holds, so that $\dim \Tame$ divides $D=2^n$. If $\dim\Tame=1$, then since there must be at most one eigenvalue of $g$ of multiplicity $2$, we see that $t=1$. The spectrum of $g$ on $\Wild$ is $\xi\mu_{2^{n}-1}$.

Now assume $\dim\Tame>1$. Then $\dim\Tame = 2^m$ for some integer $0<m<n$, and $\dim \Wild = 2^m(2^{n-m}-1)$. The spectrum of $g$ on $\Wild$ is then $\eta\mu_{2^{n-m}-1}(\mu_{2^m+1}\setminus \{1\})$ for some root of unity $\eta$ of odd order. Therefore, the $(2^{n-m}-1)(2^m+1)$th powers of $\zeta_2\zeta_3\cdots \zeta_t$ are same for all $\zeta_i\in \mu_{2^{a_i}+1}\setminus \{1\}$, so $\prod_{i=2}^{t}(2^{a_i}+1)$ divides $(2^{n-m}-1)(2^m+1)$. Also, multiplication by a primitive $(2^{n-m}-1)$th root does not change the spectrum of $g$ on $\Wild$. From the description of the spectrum of $g$, one immediately sees that if a root of unity which is not $1$ has this property, then its order divides $2^{a_1}-1$. Therefore $2^{n-m}-1$ divides $2^{a_1}-1$, and $\prod_{i=2}^{t}(2^{a_i}+1)$ divides $2^m+1$. On the other hand, if we take an eigenvalue of $g$ on $\Wild$ and multiply by $(2^m+1)$th roots of unity, then all but one of the $2^m+1$ numbers are also contained in the spectrum. But if $t\geq 3$, then strictly more than one multiple of an eigenvalue by $(2^{a_2}+1)(2^{a_3}+1)$th roots of unity are missing from the spectrum of $g$. Therefore we must have $t\leq 2$. 

If $t=2$, then we must have $2^{a_2}+1=2^m+1$, since among the multiples of an eigenvalue by the $(2^m+1)$th roots of unity, only one must be missing, but $2^{a_1}+1$ divides $2^m+1$ and one from the multiples by $(2^{a_1}+1)$th roots of unity is already missing. Hence $m=a_2$, $a_1=n-m$, $\dim\Tame=2^{a_2}$, and $\dim \Wild=(2^{a_1}-1)2^{a_2}$. Since $2^{n-m}-1$ and $2^{m}+1$ are relatively prime, we also see that $2\mid nm$.

If $t=1$, then the spectrum of $g$ on $\Wild$ is $$\{\eta\zeta_1\zeta_2\mid \zeta_1\in \mu_{2^{n-m}-1}, \zeta_2\in \mu_{2^m+1}\setminus\{1\}\}$$ for some root of unity $\eta$ of odd order, and this must be contained in $\xi\mu_{2^n-1}$. Therefore $(2^m+1)(2^{n-m}-1)=2^n+2^{n-m}-2^m-1$ divides $2^n-1$, which forces $n=2m$. Note that $(\chi|_{Q},1_{Q})=\dim \Tame= 2^m$. Also recall that $Q\cap E$ is normalized by $Q$ and $g_\infty$, $\chi$ is faithful, and the nontrivial irreducible constituents of $\chi|_Q$ are permuted cyclically by $g_\infty$. Hence, no nontrivial irreducible constituent of $\chi|_Q$ has $1_{Q\cap E}$ as a constituent when restricted to $Q\cap E$. Therefore $(\chi|_{Q\cap E},1_{Q\cap E})=2^m$. On the other hand, since $\chi|_E$ is the unique irreducible character of $E$ of degree $2^n$, $\chi|_{Q\cap E}$ is a multiple of the regular character $\rho_{Q\cap E}$ of $Q\cap E$. Since $(\rho_{Q\cap E},1_{Q\cap E})=1$, we get $\chi|_{Q\cap E} = 2^m\rho_{Q\cap E}$. In particular, $|Q\cap E|=\rho_{Q\cap E}(1) = \chi|_{Q\cap E}(1)/2^m = 2^m$. Hence, the $g$-invariant subspace $(Q\cap E)\mathbf{Z}(E)/\mathbf{Z}(E)$ of $E/\mathbf{Z}(E)$ has dimension $m$, but the $g$-invariant subspaces of $E/\mathbf{Z}(E)$ have dimension $0, n$ or $2n$. Therefore this case is impossible.
\end{proof}

\section{Determination of the finiteness of the monodromy}

From \autoref{m2sp}(a), we know the possible upstairs characters of $\mathcal{H}$, and the possible downstairs characters were determined in \autoref{g_infty}. However, not every pair of these upstairs and downstairs characters have the desired $G_{\geom}$; in fact, $G_{\geom}$ is not even finite for some of them. In this section, we determine the pairs with finite $G_{\geom}$.

By \cite[8.14.6]{K-ESDE}, the finiteness of $G_{\geom}$ is equivalent to saying that the values of trace function after a Tate twist are algebraic integers, and it is enough to check the $p$-adic valuations. This can be done by using the so-called \emph{$V$-test} of Kubert \cite[Proposition 13.2]{K-G2}\cite[Section 8.16]{K-ESDE}. 

The definition of the function $V$ can be found in \cite[Discussion after Proposition 13.1]{K-G2}; for an alternative definition, for a given prime $p$ and integers $k>0, 0\leq a<p^k-1$ and $z$, we have $$V\left(\frac{a}{p^k-1}+z\right) = \frac{\text{sum of digits of }a\text{ in base }p}{k(p-1)},$$ cf. \cite[Section 8.16]{K-ESDE}\cite[Theorem 13.4]{K-G2}. For the rest of this paper, we will use the following basic properties of the function $V$, often without mentioning. See \cite[Discussion before Proposition 13.2]{K-G2} for the proofs.
\begin{enumerate}[label=(\arabic*)]
\item $V(x)=0$ if and only if $x\in \mathbb{Z}$. 
\item $V(x)+V(-x)=1$ for any $x\notin \mathbb{Z}$.
\item $V(1/2)=V(1/2)$.
\item $V(x)=V(px)$ for any $x$.
\item $V(x)+V(y)\geq V(x+y)$ for any $x,y$. 
\item (Hasse-Davenport relation) $\sum_{i=1}^{r}V(x+i/r) = V(rx)+\frac{r-1}{2}$ for any $x$ and $p\nmid r\in \mathbb{Z}$.
\end{enumerate}

In practice, instead of the original inequality stated in \cite[Proposition 13.2]{K-G2}, the following simplified version is used. Although this version is already known to experts, it seems to be not written explicitly in literature to my knowledge, so here I include the statement and proof for the convenience of readers.
\begin{lem}\label{V-test}
Let $\chi$ be a multiplicative character of order $p^d-1$, and let $\mathcal{G}$ be the hypergeometric sheaf $\mathcal{H}yp_\psi(\chi^{a_1},\dots,\chi^{a_D}; \chi^{b_1},\dots, \chi^{b_M})$ where $a_1,\dots, a_D, b_1,\dots, b_M\in \{0,\dots, p^d-2\}$ are pairwise distinct. Then the geometric monodromy group $G_{\geom}$ of $\mathcal{G}$ is finite if and only if for every $x\in \mathbb{Q}/\mathbb{Z}$ whose denominator is not divisible by $p$ and for every $N\in \{1,\dots, p^d-2\}$ which is relatively prime to $p^d-1$, we have the inequality
\begin{align*}
\sum_{i=1}^{D}V(Na_i+x) + \sum_{j=1}^{M}V(-Nb_j-x) \geq \frac{D+M-1}{2}.
\end{align*}
\end{lem}

\begin{proof}
The only difference between this and \cite[Proposition 13.2]{K-G2} is that the term $(1/D)\sum_{i,j}V(Na_i-Nb_j)$ is replaced by $M/2$. Suppose that the inequality in this lemma holds for all $x$ and $N$. Let $N_0$ be any possible value of $N$, and choose $(N,x) = (N_0, -N_0b_j)$ and $(-N_0, N_0b_j)$ and take the sum of the resulting inequalities. Since $V(r)+V(-r) = 1$ if $r\notin \mathbb{Z}$ and $0$ otherwise, the sum of these two inequalities becomes $D + M -1 \geq D+M-1$. Therefore, the equality must hold for both pairs $(N,x)=(N_0,-N_0b_j)$ and $(-N_0,N_0b_j)$. Now take the sum of the equalities for the pairs $(N,x) = (N_0,-N_0b_1), (N_0,-N_0b_2),\dots, (N_0,-N_0b_M)$. Then we get $$\sum_{i,j}V(N_0a_i-N_0b_j) + \frac{M(M-1)}{2} = \frac{M(D+M-1)}{2} = \frac{ DM+ M(M-1)}{2}$$ so that $(1/D)\sum_{i,j}V(N_0a_i-N_0b_j) = M/2$. 

Conversely, suppose that the inequality in \cite[Proposition 13.2]{K-G2} holds for all $x$ and $N$. By an entirely analogous argument, we get $$\sum_{i,j}V(Na_i-Nb_j) + \frac{M(M-1)}{2}= \frac{M(D-1)}{2} + \frac{M}{D}\sum_{i,j}V(Na_i-Nb_j)$$ so that $(1/D)\sum_{i,j}V(Na_i-Nb_j) = (1/(D-M))(\frac{M(D-1)}{2}-\frac{M(M-1)}{2}) = M/2$.
\end{proof}

From now on, given a hypergeometric sheaf $\mathcal{G}$ as in \autoref{V-test}, we will call $$V_{\mathcal{G}, \text{up}}(N,x):=\sum_{i=1}^{D}V(Na_i+x) - D/2$$ the ``upstairs part'' of the $V$-test, and $$V_{\mathcal{G}, \text{down}}(N,x):= \sum_{j=1}^{M}V(-Nb_j-x) - M/2$$ the ``downstairs part'' of the test. Thus the inequality for the $V$-test becomes $$V_{\mathcal{G}, \text{up}}(N,x)+V_{\mathcal{G}, \text{down}}(N,x) +\frac{1}{2} \geq 0.$$

Now we return to the notations in the previous section, so $\mathcal{H}$ is an irreducible hypergeometric sheaf of rank $D=2^n$ in characteristic $2$ whose $G_{\geom}$ is an extraspecial normalizer. The element $g_0$ is as described in \autoref{m2sp}(a); let $a_1>\cdots> a_t$ be the positive integers in this description, so that $a_1+\cdots + a_t = n$, and the numbers $A_i:=2^{a_i}+1$ are pairwise coprime. The upstairs characters are the multiplicative characters of order dividing $A_1A_2\cdots A_t$ but not dividing $A_1A_2\cdots A_t/A_i$ for any $i$, tensored with a common multiplicative character. Therefore, the $V$-test of this $\mathcal{H}$ has the upstairs part
\begin{align*}
&V_{\mathcal{H},\text{up}}(N,x)=\sum_{i_1 = 1}^{2^{a_1}}\sum_{i_2=1}^{2^{a_2}}\cdots \sum_{i_t = 1}^{2^{a_t}}V\left(N\left(\frac{i_1}{A_1}+\cdots + \frac{i_t}{A_t}\right) + x\right)-\frac{D}{2}\\=& \sum_{i_1,\dots,i_{t-1}}\left(V\left(A_tN\left(\frac{i_1}{A_1}+\cdots + \frac{i_{t-1}}{A_{t-1}}\right)+A_tx\right) + 2^{a_t-1} - V\left(N\left(\frac{i_1}{A_1}+\cdots + \frac{i_{t-1}}{A_{t-1}}\right)+x\right) \right)-\frac{D}{2} \\=& \cdots \\=& \sum_{i=0}^{t} \left((-1)^{t-i}\sum_{1\leq t_1<\cdots < t_i\leq t}V(A_{t_1}A_{t_2}\cdots A_{t_i}x)\right)
\end{align*}
where the equalities follow from the Hasse-Davenport relation stated above. 

Let us write $$V(2; r_1, \dots, r_t; x) := \sum_{i=0}^{t} \left((-1)^{t-i}\sum_{1\leq t_1<\cdots < t_i\leq t}V((2^{r_{t_1}}+1)\cdots (2^{r_{t_i}}+1)x)\right)$$ so that $V_{\mathcal{H},\text{up}}(N,x) = V(2;a_1,\dots,a_t;x)$.
We now prove several inequalities involving this function.

\begin{lem}\label{V_lem}
Let $r\in \mathbb{Z}_{>0}$. In characteristic $2$, the following hold.
\begin{enumerate}[label={\upshape{(\alph*)}}]
\item $V((2^r+1)x) \leq 2V(x) \leq V((2^r+1)x)+1$. 
\item If $V((2^r+1)x) = 0$, then $V(x)=0$ or $1/2$. 
\end{enumerate}
\end{lem}
\begin{proof}
(a) If $x\in \mathbb{Z}$, then $V((2^r+1)x) = V(x) = 0$. If $x\notin \mathbb{Z}$, then we have $V((2^r+1)x) \leq V(2^rx)+V(x) = 2V(x) = 2-2V(-x) \leq 2 - V(-(2^r+1)x) = V((2^r+1)x) +1$. 

(b) Since $V((2^r+1)x)=0$, we must have $(2^r+1)x\in \mathbb{Z}$. If $x\in \mathbb{Z}$, then $V(x)=0$. So suppose $x\notin \mathbb{Z}$. Then $V(x) = 1-V(-x) = 1-V((2^r+1)x - x) = 1-V(2^rx) = 1-V(x)$, so $V(x)=1/2$.
\end{proof}

\begin{crl}\label{1-piece}
In characteristic $2$, $V(2; r_1; x)= V((2^r+1)x) - V(x)\in [-1/2, 1/2)$ for every positive integer $r$. It is $-1/2$ if and only if $x = i/(2^r+1)$ for some integer $i$ not divisible by $2^r+1$.
\end{crl}
\begin{proof}
If $V((2^r+1)x)= 0$ then $V(2;r_1;x) = 0$ or $-1/2$ by \autoref{V_lem}(b). So assume $0<V((2^r+1)x)<1$. Then $$-V(x) < V((2^r+1)x)-V(x) < 1- V(x).$$ On the other hand, from \autoref{V_lem}(a) we get $$V(x)-1\leq V((2^r+1)x)-V(x) \leq V(x).$$ Combining these two inequalities, we get $V(2;r_1;x)\in (-1/2, 1/2)$.
\end{proof}

\begin{crl}
\begin{enumerate}[label={\upshape{(\alph*)}}]
\item If $g_\infty$ is as in \autoref{g_infty}(a), then $V_{\mathcal{H}, \text{down}}(N,x)=0$.
\item If $g_\infty$ is as in \autoref{g_infty}(b), then $V_{\mathcal{H}, \text{down}}(N,x)=V(-(2^{a_2}+1)x)-V(-x) \in [-1/2,1/2)$. Also, $V_{\mathcal{H}, \text{down}}(N,x)=-1/2$ precisely when $x=i/(2^{a_2}+1)$ for some integer $i$ not divisible by $2^{a_2}+1$.
\item If $g_\infty$ is as in \autoref{g_infty}(c), then $V_{\mathcal{H}, \text{down}}(N,x)=V(-x)-1/2\in [-1/2,1/2)$.
\end{enumerate}
\end{crl}

\begin{proof}
(a) and (c) are obvious. For (b), the downstairs part is $V(-(2^{a_2}+1)x)-V(-x) = V(2;a_2;-x)$, so we can apply \autoref{1-piece}.
\end{proof}

\begin{lem}\label{2-pieces}
In characteristic $2$, $V(2; r_1,r_2;x)= V((2^{r_1}+1)(2^{r_2}+1)x) - V((2^{r_1}+1)x)-V((2^{r_2}+1)x) + V(x)\in [-3/4, 3/4)$ for every positive integers $r_1, r_2$ such that $2^{r_1}+1$ and $2^{r_2}+1$ are coprime. It is $-3/4$ if and only if $(2^{r_1}+1)(2^{r_2}+1)x\in \mathbb{Z}$ and $V(x)=1/4$. 
\end{lem}

\begin{proof}
Let $X:=V(2;r_1, r_2;x) = V((2^{r_1}+1)(2^{r_2}+1)x) - V((2^{r_1}+1)x)-V((2^{r_2}+1)x) + V(x)$. First note that 
\begin{align*}
&V((2^{r_1}+1)(2^{r_2}+1)x)  \\\leq& \frac{1}{2}(V(2^{r_1}(2^{r_2}+1)x) + V((2^{r_2}+1)x) + V(2^{r_2}(2^{r_1}+1)x) + V((2^{r_1}+1)x)) \\=& V((2^{r_2}+1)x) + V((2^{r_1}+1)x)
\end{align*}
so that $X \leq V(x)$. 

If $(2^{r_1}+1)(2^{r_2}+1)x\in \mathbb{Z}$, then $V((2^{r_1}+1)x), V((2^{r_2}+1)x) \in \{0, 1/2\}$. If both of them are $0$, then $x\in \mathbb{Z}$, so $X=0$. If one of them is $0$, then $V(x)=1/2$ and $X=0$. If both of them are $1/2$, then 
\begin{align*}
\frac{1}{4} &= \frac{1}{2}V((2^{r_1}+1)x) \leq \frac{1}{2}(V(2^{r_1}x) + V(x)) = V(x) = 1-V(-x) \\&= 1- \frac{1}{2}(V(-x) + V(-2^{r_1}x)) \leq 1-\frac{1}{2}V(-(2^{r_1}+1)x) = \frac{3}{4}
\end{align*}
so $X \in [-3/4, -1/4]$.

Suppose that $(2^{r_1}+1)(2^{r_2}+1)x\notin \mathbb{Z}$. Then 
\begin{align*}
&V((2^{r_2}+1)x) + V((2^{r_1}+1)x) = 2-V(-(2^{r_2}+1)x) - V(-(2^{r_1}+1)x) \\\leq& 2-V(-(2^{r_1}+1)(2^{r_2}+1)x) = 1+V((2^{r_1}+1)(2^{r_2}+1)x)
\end{align*}
so that $X\geq V(x)-1$.

Also, 
\begin{align*}
2V(x)-1&= 1- V(-2^{r_1}x)-V(-x) \leq 1-V(-(2^{r_1}+1)x) \\&= V((2^{r_1}+1)x) \leq V(2^{r_1}x)+V(x) = 2V(x)
\end{align*}
and similarly $2V(x)-1 \leq V((2^{r_2}+1)x) \leq 2V(x)$, so $-3V(x) < X < -3V(x)+3$. Combining these results, we can see that $X \in [-3/4, 3/4)$, and $X=-3/4$ happens only when $(2^{r_1}+1)(2^{r_2}+1)x\in \mathbb{Z}$ and $V(x) = 1/4$. 
\end{proof}

For $a_1,\dots, a_t$ such that $2^{a_j}+1$ are pairwise coprime, we will write $$x_{a_1,\dots,a_t}(i_1,\dots,i_t) := \sum_{j=1}^{t}\frac{i_j}{2^{a_j}+1}.$$ If the numbers $a_1,\dots,a_t$ are clear from the context, we will omit the subscript and simply write $x(i_1,\dots,i_t)$. 

\begin{lem}\label{induction_lemma}
For any positive integers $a_1,\dots, a_t$ with $t\geq2$ such that $(2^{a_j}+1)$ are pairwise coprime, at least one of the following holds.
\begin{enumerate}[label={\upshape{(\roman*)}}]
\item There exist integers $i_1,\dots, i_t$ such that $V(2; a_1,\dots,a_t; x(i_1,\dots, i_t))\neq -1/2$ and for every $j$, $(2^{a_j}+1)\nmid i_j$.
\item There exist some $j_0\in \{1,\dots,t\}$ and integers $i_1,\dots, i_t$ such that $(2^{a_{j_0}}+1)\mid i_{j_0}$ and $V(2; a_1,\dots, a_t; x(i_1,\dots, i_t))\notin [-1/2,1/2]$.
\end{enumerate}

\end{lem}

\begin{proof}
We use induction on $t$. 

(1) First assume that $t=2$. We show that (i) holds in this case by finding explicit integers $i_1, i_2$. Let $b_j$ be the $2$-part of $a_j$. Then $2^{a_j}+1 = (2^{b_j}+1)(2^{a_j-b_j}-2^{a_j-2b_j}+\cdots -2^{b_j}+1)$. Since $2^{a_j}+1$ are pairwise coprime, we must have $b_1\neq b_2$. We may assume that $b_1>b_2$. Let $$i_1= \frac{2^{a_1}+1}{2^{b_1}+1}(2^{b_1/2}-1)(2^{b_1/4}-1)\cdots (2^{2b_2}-1)(2^{b_2}-1)\text{ and }i_2=\frac{2^{a_2}+1}{2^{b_2}+1}.$$ Then $2^{a_j}+1$ does not divide $i_j$ for both $j=1,2$, and
\begin{align*}
&V\left( \frac{i_1}{2^{a_1}+1}+\frac{i_2}{2^{a_2}+1}\right) = V\left(\frac{(2^{b_{1}/2}-1)(2^{b_{1}/4}-1)\cdots(2^{2b_2}-1)(2^{b_2}-1) }{2^{b_{1}}+1}+\frac{1}{2^{b_2}+1}\right) \\=&  V\left(\frac{(2^{b_{1}}-1)(2^{b_{1}/2}-1)\cdots(2^{b_2}-1) + (2^{b_{1}}+1)(2^{b_{1}/2}+1)\cdots (2^{2b_2}+1)(2^{b_2}-1) }{2^{2b_{1}}-1}\right)\\=&  V\left(\frac{(2^{b_2}-1)B }{2^{2b_{1}}-1}\right)
\end{align*}
where $$B:=(2^{b_{1}}-1)(2^{b_{1}/2}-1)\cdots(2^{2b_2}-1) + (2^{b_{1}}+1)(2^{b_{1}/2}+1)\cdots (2^{2b_2}+1).$$ This number is just the sum of all numbers of the form $2^{1+(2^{k_1}+\cdots+2^{k_r})b_2}$ for some $1\leq k_r<\cdots < k_1\leq \log_2 (b_{1}/b_2)$ with $r\equiv \log_2 (b_{1}/b_2)$ mod $2$. These numbers are distinct for different choice of numbers $k_1,\dots,k_r$, and the ratio of two of them is at least $2^{b_2}$. Therefore, if we write $B$ in base $2$, the digits $1$ are always followed by $b_2-1$ digits $0$; in other words, there are at least $b_2-1$ $0$'s between any two $1$'s. Thus multiplying $B$ with $2^{b_{2}}-1$, which is just $b_{2}$ consecutive $1$'s when written in base $2$, is the same as replacing each occurrence of $b_2$ consecutive base $2$ digits $0\dots 01$ in the base $2$ representation of $B$ by $1\dots 1$. Since there are exactly $2^{\log_2(b_{1}/b_2)-1}=(b_{1}/2b_2)$ distinct choices of $k_1,\dots, k_r$, there are the same number of $b_2$ consecutive digits $0\dots 01$. Therefore $(2^{b_2}-1)B$ has exactly $b_{1}/2$ nonzero base $2$ digits, and since $(2^{b_2}-1)B<2^{2b_1}-1$, we have
\begin{align*}
&V\left(\frac{(2^{b_2}-1)B}{2^{2b_{t-1}}-1}\right) = \frac{b_1/2}{2b_{1}} = \frac{1}{4}\text{, so that }V(2; a_1,a_2;x(i_1,i_2))
\end{align*}
Therefore, (i) holds when $t=2$.

(2) Now assume that the statement holds for $t=t_0-1$ so that at least one of (i) and (ii) holds, and consider the case $t=t_0$. Assume that $a_1>a_2>\cdots > a_{t_0}$, and suppose that both (i) and (ii) are not true. Then for any $i_2,\dots, i_{t_0}$ such that $(2^{a_j}+1)\nmid i_j$ for every $j\in \{2,\dots, t_0\}$, we have
\begin{align*}
&2^{a_1}V(2;a_2,\dots,a_{t_0}; (2^{a_1}+1)x(0,i_2,\dots,i_{t_0}))\\ =& (2^{a_1}+1)V(2;a_2,\dots,a_{t_0}; (2^{a_1}+1)x(0,i_2,\dots,i_{t_0})) - V(2;a_2,\dots,a_{t_0};(2^{a_1}+1)x(0,i_2,\dots,i_{t_0}))\\=&\sum_{i=0}^{2^{a_1}}V(2;a_2,\dots,a_{t_0}; (2^{a_1}+1)x(i,i_2,\dots,i_{t_0})) -\sum_{i=0}^{2^{a_1}} V(2;a_2,\dots,a_{t_0}; x(i,i_2,\dots,i_{t_0})) \\=&\sum_{i=0}^{2^{a_1}}V(2;a_1,\dots,a_{t_0}; x(i, i_2,\dots,i_{t_0}))\\=&-\frac{2^{a_1}}{2} + V(2; a_1,\dots,a_{t_0}; x(0,i_2,\dots, i_{t_0})) \in \left[ -2^{a_1-1}-\frac{1}{2}, -2^{a_1-1}+\frac{1}{2}\right].
\end{align*} 
Here, the second equality follows from the Hasse-Davenport relation stated in the beginning of this section, and the last equality and containment follow from the assumption that both (i) and (ii) are not true.

On the other hand, $2^{2\lcm(a_2,\dots,a_{t_0})}-1$ is a common denominator of $\{i_j/(2^{a_j}+1)\mid j=2,\dots,t_0\}$ and $x(0,i_2,\dots, i_{t_0})$ is the sum of them, so by the definition of the function $V$, we have $$V(2;a_2,\dots, a_{t_0};(2^{a_1}+1)x(0,i_2,\dots,i_{t_0}))\in \left(\frac{1}{2\text{lcm}(a_2,\dots,a_{t_0})}\right)\mathbb{Z}.$$ Note that since $2^{a_1}+1,\dots, 2^{a_{t_0}}+1$ are pairwise coprime, as we saw above, the $2$-parts of $a_j$ are distinct, so in particular $a_1\geq 2^{t_0-1}$. Hence if $a_1\geq 16$, then we have $$2^{a_1} \geq 2^{(\log_2 a_1)^2}\geq 2^{(t_0-1)\log_2 a_1}=(a_1)^{t_0-1} > \prod_{j=2}^{t_0}a_{t_0}\geq \text{lcm}(a_2,\dots,a_{t_0}).$$ One can also easily check that $2^{a_1}> \lcm(a_2,\dots,a_{t_0})$ when $a_1<16$. Therefore, $$V(2;a_2,\dots,a_{t_0}; (2^{a_1}+1)x(0,i_2,\dots,i_{t_0}))\in \left(\frac{1}{2\text{lcm}(a_2,\dots,a_{t_0})}\right)\mathbb{Z}\cap\left[ -\frac{1}{2}-\frac{1}{2^{a_1+1}}, -\frac{1}{2}+\frac{1}{2^{a_1+1}}\right] = \left\lbrace\frac{1}{2}\right\rbrace. $$ Therefore, $a_2,\dots,a_{t_0}$ does not satisfy (i).

For any $j_0\in \{2,\dots, t_0\}$ and $i_2,\dots,i_{t_0}$ such that $i_j$ is not divisible by $2^{a_j}+1$ if and only if $j_0\neq j$, we have
\begin{align*}
&(2^{a_1}+1)V(2;a_2,\dots,a_{t_0}; (2^{a_1}+1)x(0,i_2, i_3,\dots,i_{t_0}))\\=&\sum_{i=0}^{2^{a_1}}V(2;a_2,\dots,a_{t_0}; (2^{a_1}+1)x(i,i_2, i_3,\dots,i_{t_0})) \\=&\sum_{i=0}^{2^{a_1}}\left( V(2; a_1,\dots, a_{t_0}; x(i,i_2,i_3\dots, i_{t_0})) + V(2; a_2,\dots,a_{t_0}; x(i,i_2,i_3,\dots,i_{t_0}))\right)\\=&\sum_{i=0}^{2^{a_1}} V(2; a_1,\dots, a_{t_0}; x(i,i_2,i_3\dots, i_{t_0})) + V(2; a_2,\dots,a_{t_0}; (2^{a_1}+1)x(0,i_2,i_3,\dots,i_{t_0})).
\end{align*}
As before, we get 
\begin{align*}
V(2;a_2,\dots,a_{t_0}; (2^{a_1}+1)x(0,i_2, i_3,\dots,i_{t_0})) \in& \left[-\frac{2^{a_1}+1}{2^{a_1+1}}, \frac{2^{a_1}+1}{2^{a_1+1}} \right] \cap \left(\frac{1}{2\text{lcm}(a_3,\dots,a_{t_0})}\right)\mathbb{Z} \\&= \left[-\frac{1}{2}, \frac{1}{2}\right].
\end{align*}
Therefore (ii) also fails for $a_2,\dots,a_{t_0}$. This contradicts the induction hypothesis that at least one of (i) and (ii) holds for $t=t_0-1$. Therefore, at least one of (i) and (ii) must also hold for $a_1,\dots,a_{t_0}$. 
\end{proof}

Now we are ready to prove the main result of this paper.

\begin{thm}\label{mainthm}
The irreducible hypergeometric sheaves in characteristic $2$ of type $(D,M)$ with $D>M$ whose geometric monodromy group is an extraspecial normalizer are the following, up to tensoring with a Kummer sheaf.
\begin{enumerate}[label={\upshape{(\alph*)}}]
\item $\mathcal{H}yp_\psi(\Char(2^a+1)\setminus \{\mathds{1}\}; \emptyset)$ for some positive integer $a$.
\item $\mathcal{H}yp_\psi(\Char(2^a+1)\setminus \{\mathds{1}\}; \Char(2^b+1)\setminus\{\mathds{1}\})$ for some positive integers $a>b$ with different $2$-parts.
\item $\mathcal{H}yp_\psi((\Char(2^a+1)\setminus \{\mathds{1}\})\times (\Char(2^b+1)\setminus\{\mathds{1}\}); \mathds{1})$ for some positive integers $a$ and $b$ with different $2$-parts, and $a+b\geq 5$.
\end{enumerate}
\end{thm}
\begin{proof}
Let $g_0$ be as in \autoref{m2sp}(a), and let $t$ and $a_1>\dots>a_t$ be as defined there. First suppose that $t\geq 3$.

(1) If $g_\infty$ is as in \autoref{g_infty}(a), then the $V$-test inequality is 
\begin{equation}\label{Vtest_Kloosterman}
V(2;a_1,\dots,a_t;x) + \frac{1}{2}\geq 0.
\end{equation}
Suppose that this inequality holds for all $x$. As before, let $x(i_1,\dots,i_t) := \frac{i_1}{2^{a_1}+1} + \cdots + \frac{i_t}{2^{a_t}+1}$ for integers $i_j$. From the properties (1) and (2) of the function $V$ stated before \autoref{V-test}, we see that
\begin{equation}\label{equal_to_1/2_Kloosterman}
V(2;a_1,\dots,a_t; x(i_1,\dots,i_t)) + V(2; a_1,\dots, a_t; -x(i_1, \dots,i_t)) = -1\text{ if }\forall j,\ 2^{a_j}+1\nmid i_j.
\end{equation} As we assumed that \eqref{Vtest_Kloosterman} holds for all $x$, it follows that 
\begin{align*}
V(2;a_1,\dots,a_t;x(i_1,\dots,i_t)) = -\frac{1}{2}\text{ if no }i_j\text{ is divisible by }2^{a_j}+1.
\end{align*} If $i_j$ is divisible by $2^{a_j}+1$ for some $j$, then we instead get 
$$V(2;a_1,\dots,a_t; x(i_1,i_2,\dots,i_t)) + V(2; a_1,\dots, a_t; x(-i_1, -i_2, \dots,-i_t)) = 0,$$ so from \eqref{Vtest_Kloosterman} we get 
\begin{equation}\label{in_range_Kloosterman}
V(2;a_1,\dots,a_t;x(i_1,i_2,\dots,i_t)) \in \left[-\frac{1}{2}, \frac{1}{2}\right].
\end{equation} However, since we are assuming $t\geq 3$, by \autoref{induction_lemma} at least one of \eqref{equal_to_1/2_Kloosterman} and \eqref{in_range_Kloosterman} must fail. Therefore \eqref{Vtest_Kloosterman} cannot hold for all $x$.

(2) Suppose that $g_\infty$ is as in \autoref{g_infty}(b). Write $\dim\Tame=2^b$ for some integer $1\leq b<n$. The $V$-test inequality is 
\begin{equation}\label{Vtest_ntriv}
V(2;a_1,\dots,a_t;x) + V(2;b;-x) + \frac{1}{2}\geq 0.
\end{equation}
By taking the sum of \eqref{Vtest_ntriv} for $x=x(i_1,\dots,i_t)$ and $x(-i_1,\dots,-i_t)$ defined as before, we again see that 
\begin{align*}
V(2;a_1,\dots,a_t;x(i_1,\dots,i_t))+V(2;b;-x(i_1,\dots,i_t))\begin{cases}=-\frac{1}{2}\text{ if } (2^{a_j}+1)\nmid i_j\forall j,\\\in [-1/2, 1/2]\text{ otherwise}  \end{cases}
\end{align*}
and that $(2^b+1)x(i_1,\dots,i_t)\notin \mathbb{Z}$ if $(2^{a_j}+1)\nmid i_j$ for all $j$. Hence there exists at least one $j$ such that $2^{a_j}+1$ does not divide $2^b+1$. If $(2^{a_j}+1)\nmid i_j$ for all $j$, then the sum of the left-hand side of the above equalities for $i_1=0,\dots, 2^{a_1}$ can be written as 
\begin{align*}
&\sum_{i=0}^{2^{a_1}}\left(V(2; a_1,\dots, a_t; x(i,i_2,\dots,i_t)) +V(2;b;-x(i,i_2,\dots,i_t))   \right) \\=&2^{a_1}V(2; a_2,\dots, a_t; (2^{a_1}+1)x(0,i_2,\dots, i_t)) +V(2;b;-(2^{a_1}+1)x(0,i_2,\dots,i_t)).
\end{align*}
By \autoref{2-pieces} applied to $V(2; b, a_1; -x(i,i_2,\dots,i_t))=V(2;b;-(2^{a_1}+1)x(0,i_2,\dots, i_t))-V(2;b;-x(0,i_2,\dots, i_t))$,
\begin{align*}
&V(2; a_2,\dots, a_t; (2^{a_1}+1)x(0,i_2,\dots, i_t))\\=&-\frac{1}{2} + 2^{-a_1}V(2; a_1,\dots, a_t; x(0,i_2,\dots,i_t)) -2^{-a_1}V(2; b, a_1; -x(0,i_2,\dots, i_t)) \\\in& \left[-\frac{1}{2}-\frac{1}{2^{a_1+1}}- \frac{3}{2^{a_1+2}}, -\frac{1}{2}+\frac{1}{2^{a_1+1}}+ \frac{3}{2^{a_1+2}}  \right] = \left[-\frac{1}{2}-\frac{5}{2^{a_1+2}}, -\frac{1}{2}+\frac{5}{2^{a_1+2}}  \right].
\end{align*}
By a similar argument as in the proof of \autoref{induction_lemma}, we see that $\lcm(a_2,\dots,a_t) < 2^{(\log_2 a_1)^2}< 2^{a_1+2-\log_2 5}$ for $a_1\geq 18$, and one can check that for $a_1<18$ we also have $\lcm(a_2,\dots,a_t)<2^{a_1+2}/5$. Therefore we have
\begin{equation}\label{equal_1/2_ntriv}
\begin{split}
&V(2; a_2,\dots, a_t; (2^{a_1}+1)x(0,i_2,\dots, i_t))\\&\in\left(\frac{1}{2\text{lcm}(a_2,\dots,a_t)}\right)\mathbb{Z}\cap\left[ -\frac{1}{2}-\frac{5}{2^{a_1+2}}, -\frac{1}{2}+\frac{5}{2^{a_1+2}}\right] = \left\lbrace-\frac{1}{2}\right\rbrace. 
\end{split}
\end{equation} Also, if $(2^{a_{j}}+1)\mid i_{j}$ for some $j\in \{2,\dots,t\}$, then as before we have 
\begin{align*}
& 2^{a_1}V(2;a_2,\dots,a_t; (2^{a_1}+1)x(0,i_2,i_3,\dots,i_t)) + V(2;b; -(2^{a_1}+1)x(0,i_2,i_3,\dots,i_t))\\=&\sum_{i=0}^{2^{a_1}}\left(V(2;a_1,\dots,a_t;x(i,i_2,i_3,\dots,i_t)) + V(2;b;-x(i,i_2,i_3,\dots,i_t))  \right)\\\in& \left[ -\frac{2^{a_1}+1}{2}, \frac{2^{a_1}+1}{2}\right]
\end{align*}
so by \autoref{2-pieces} applied to $V(2; b, a_1; -x(i,i_2,\dots,i_t))=V(2;b;-(2^{a_1}+1)x(0,i_2,\dots, i_t))-V(2;b;-x(0,i_2,\dots, i_t))$, we get 
\begin{equation}\label{in_range_ntriv}
\begin{split}
&V(2;a_2,\dots,a_t; (2^{a_1}+1)x(0,i_2,i_3,\dots,i_t)) \\\in&  \left[ -\frac{1}{2}- \frac{5}{2^{a_1+2}}, \frac{1}{2} + \frac{5}{2^{a_1+2}}\right] \cap \left(\frac{1}{2\text{lcm}(a_2,\dots,a_t)}\right)\mathbb{Z} = \left[-\frac{1}{2}, \frac{1}{2}\right].
\end{split}
\end{equation}
As before, by \autoref{induction_lemma}, one of \eqref{equal_1/2_ntriv} and \eqref{in_range_ntriv} must fail, so \eqref{Vtest_ntriv} fails for some $x$.

(3) Finally suppose that $g_\infty$ is as in \autoref{g_infty}(c). In this case $E$ is of type $+$, so $t$ is even. The $V$-test inequality is 
\begin{equation}\label{Vtest_triv}
V(2;a_1,\dots,a_t;x) + V(-x)\geq 0.
\end{equation}
We again get 
\begin{align*}
V(2;a_1,\dots,a_t;x(i_1,\dots,i_t))+V(-x(i_1,\dots,i_t))-\frac{1}{2} \begin{cases}=-\frac{1}{2}\text{ if } (2^{a_j}+1)\nmid i_j\forall j,\\\in [-1/2, 1/2]\text{ otherwise}.  \end{cases}
\end{align*}
If $(2^{a_j}+1)\nmid i_j$ for $j=2,\dots,t$, then the sum of the left-hand side of the above inequalities for $i_1=0,\dots, 2^{a_1}$ is
\begin{align*}
&2^{a_1}V(2;a_2,\dots,a_t; (2^{a_1}+1)x(0,i_2,\dots, i_t)) + V(-(2^{a_1}+1)x(0,i_2,\dots,i_t)) -\frac{1}{2} \\=&\sum_{i=0}^{2^{a_1}}\left(V(2;a_1,\dots,a_t;x(i,\dots,i_t))+V(-x(i,\dots,i_t))-\frac{1}{2} \right) \in \left[-2^{a_1-1}-\frac{1}{2}, -2^{a_1-1}+\frac{1}{2} \right] .
\end{align*}
Therefore we get 
\begin{align*}
V(2;a_2,\dots,a_t; (2^{a_1}+1)x(0,i_2,\dots, i_t)) \in \left(-\frac{1}{2} -\frac{1}{2^{a_1}} , -\frac{1}{2} +\frac{1}{2^{a_1}} \right] \cap \left( \frac{1}{2 \lcm(a_2,\dots,a_t)}\right)\mathbb{Z} .
\end{align*}
If $a_1\geq 20$, then $\lcm(a_2,\dots,a_t) < 2^{(\log_2 a_1)^2}< 2^{a_1-1}$, and one can check that we still have $\lcm(a_2,\dots,a_t)<2^{a_1-1}$ for $a_1<20$, unless $t=4$ and $(a_1,a_2,a_3,a_4) = (8,7,6,4)$. Therefore, assuming $(a_1,a_2,a_3,a_4)\neq (8,7,6,4)$, we must have
\begin{equation}\label{equal_1/2_triv}
V(2;a_2,\dots,a_t; (2^{a_1}+1)x(0,i_2,\dots, i_t))=-1/2.
\end{equation}

For $i_2,\dots, i_t$ with $(2^{a_j}+1)\mid i_j$ for some $j$, we have 
\begin{align*}
&2^{a_1}V(2;a_2,\dots,a_t; (2^{a_1}+1)x(0,i_2,\dots, i_t)) + V(-(2^{a_1}+1)x(0,i_2,\dots,i_t)) -\frac{1}{2} \\=&\sum_{i=0}^{2^{a_1}}\left(V(2;a_1,\dots,a_t;x(i,\dots,i_t))+V(-x(i,\dots,i_t))-\frac{1}{2} \right) \in \left[-2^{a_1-1}-\frac{1}{2}, 2^{a_1-1}+\frac{1}{2} \right] .
\end{align*}
so if $(a_1,a_2,a_3,a_4)\neq (8,7,6,4)$, then 
\begin{equation}\label{in_range_triv}
\begin{split}
&V(2;a_2,\dots,a_t; (2^{a_1}+1)x(0,i_2,\dots, i_t)) \\&\in \left( -\frac{1}{2}-\frac{1}{2^{a_1}}, \frac{1}{2} + \frac{1}{2^{a_1}}    \right]\cap \left( \frac{1}{2 \lcm(a_2,\dots,a_t)}\right)\mathbb{Z} = \left[-\frac{1}{2},\frac{1}{2}\right] .
\end{split}
\end{equation}
By \autoref{induction_lemma}, not both \eqref{equal_1/2_triv} and \eqref{in_range_triv} can be true at the same time. Therefore \eqref{Vtest_triv} fails unless $t=4$ and $(a_1,a_2,a_3,a_4)=(8,7,6,4)$. For this case, one can quickly find values of $x$ for which \eqref{Vtest_triv} fails using computers.

(4) We proved that $t<3$. For $t=1$, $E$ is of type $-$, so $g_{\infty}$ must be as in \autoref{g_infty}(a) (the Pink-Sawin Kloosterman sheaf \cite[Theorem 7.3.8]{KT4}) or \autoref{g_infty}(b) (studied in \cite[Theorem 8.5.5]{KT4} when $a_1\geq 4$ and in \cite[Theorem 4.4]{KT6} when $a_1<4$).

For $t=2$, $E$ is of type $+$, so $g_{\infty}$ must be as in \autoref{g_infty}(c) (studied in \cite[Proposition 9.1.9]{KT4} when $a+b\geq 4$) or \autoref{g_infty}(d). For $g_{\infty}$ as in \autoref{g_infty}(c), the condition $a_1+a_2\geq 4$ is necessary here, because the case $a_1+a_2=3$ has $G_{\geom}= 2\mathsf{A}_8$ as shown in \cite[Theorem 9.1.8]{KT4}. For \autoref{g_infty}(d), we must have $a_1=a_2=1$, which is impossible. 
\end{proof}

\subsection*{Acknowledgment}
I would like to thank Pham Huu Tiep for suggesting this problem and giving helpful advice.

\end{document}